\documentclass[11pt]{amsart}

\usepackage{amsmath}
\usepackage{amssymb}
\usepackage{graphicx}

\newtheorem{theorem}{Theorem}[section]

\theoremstyle{definition}

\numberwithin{equation}{section}

\title[On a necessary condition]{On a necessary condition for removing singularities of solutions of nonlinear elliptic inequalities}

\author[A. Kon'kov]{Andrej Kon'kov}

\address[A. Kon'kov]{Department of Differential Equations,
Faculty of Mechanics and Mathematics,
Mo\-s\-cow Lo\-mo\-no\-sov State University,
Vorobyovy Gory,
Moscow, 119992 Russia}
\email{{\tt konkov@mech.math.msu.su}}

\author[A. Shishkov]{Andrey Shishkov}
\address[A. Shishkov]{Center of Nonlinear Problems of Mathematical Physics,
RUDN University,
Mi\-klu\-k\-ho-Mak\-la\-ya str. 6,
Moscow, 117198 Russia}
\email{\tt shishkov-ae@rudn.ru}



\begin{document}

\begin{abstract}
We study solutions of the differential inequality
$$
	\Delta^{m / 2} u \ge f (x) g (u)
	\quad
	\mbox{in } B_1 \setminus \{ 0 \},
$$
where $m \ge 2$ is an even integer, $f$ and $g$ are some functions, and $B_1$ is an open unit ball in $R^n$, $n \ge 2$, centered at zero.
Our aim is to obtain a necessary condition for a singularity at zero to be removable for any solution of this inequality.
\end{abstract}

\maketitle


\section{Introduction}

We deal with differential inequalities of the form
\begin{equation}
	\Delta^{m / 2} u \ge f (x) g (u)
	\quad
	\mbox{in } B_1 \setminus \{ 0 \},
	\label{1.1}
\end{equation}
where $m \ge 2$ is an even integer. It is assumed that $f \in L_{loc, \infty} (\overline{B_1} \setminus \{ 0 \})$ and, moreover, $g : (0, \infty) \to (0, \infty)$ is a non-decreasing continuous function.
As is customary, by $B_r$ we mean the open ball in $\mathbb R^n$, $n \ge 2$, of radius $r > 0$ centered at zero.

A function $u \in L_1 (B_1 \setminus \{ 0 \})$ is called a solution of~\eqref{1.1} if $f (x) {g (u)} \in L_1 (B_1 \setminus \{ 0 \})$ and
\begin{equation}
	\int_{B_1}
	u
	\Delta^{m / 2} \varphi
	\,
	dx
	\ge
	\int_{B_1}
	f (x)
	g (u)
	\varphi
	\,
	dx
	\label{1.2}
\end{equation}
for all non-negative functions $\varphi \in C_0^\infty (B_1 \setminus \{ 0 \})$.
In so doing, we say that $u$ has a removable singularity at zero if
$u \in L_1 (B_1)$,
$f (x) {g (u)} \in L_1 (B_1)$, 
and inequality~\eqref{1.2} is valid for all non-negative functions $\varphi \in C_0^\infty (B_1)$.

In paper~\cite{KSFCAA}, it is arisen the condition
\begin{equation}
	\int_1^\infty
	g^{- 1 / m} (\zeta)
	\zeta^{1 / m - 1}
	\,
	d\zeta
	<
	\infty
	\label{1.3}
\end{equation}
which was required for singularity at zero to be removable for any solution of~\eqref{1.1}.
For $m = 2$, it is easy to see that~\eqref{1.3} is equivalent to the well-known Keller--Osserman blow-up condition~\cite{Keller, Osserman, KSNonlinearity}.
For $m = 2$, relation~\eqref{1.3} also appeared in~\cite{meCMFI}, where the differential inequality
\begin{equation}
	\Delta u
	\ge
	a (x, u)
	\quad
	\mbox{in } B_1 \setminus \{ 0 \}
	\label{1.4}
\end{equation}
was considered. 
Assume that $a$ is a non-negative function such that
$$
	q(r) g(\zeta) 
	\le 
	\operatorname*{ess\,inf}_{
		x \in B_{2 r} \setminus B_{r / 2}
	}
	a (x, r^{2 - n} \zeta)
$$
with a measurable function $q : (0, 1 / 2] \to [0, \infty)$ and a non-decreasing continuous function $g : (0, \infty) \to (0, \infty)$ for all real numbers $0 < r \le 1 / 2$ and $\zeta > 1$.
As shown in~\cite[Theorems~6.1.1 and~6.1.5]{meCMFI}, any positive solution of~\eqref{1.4} has a removable singularity at zero if
\begin{equation}
	\int_1^\infty
	(g (\zeta) \zeta)^{- 1 / 2}
	\,
	d\zeta
	<
	\infty,
	\label{1.5}
\end{equation}
\begin{equation}
	\int_0^{1/2}
	r^{n - 1}
	q (r)
	dr
	=
	\infty,
	\label{1.6}
\end{equation}
and
\begin{equation}
	\int_0^{1/2}
	r^{n - 1}
	\operatorname*{ess\,inf}_{
		x \in B_{2 r} \setminus B_{r / 2}
	}
	a (x, r^{2 - n} \varepsilon)
	dr
	=
	\infty
	\label{1.7}
\end{equation}
for all real numbers $\varepsilon > 0$.

In the case of the inequality
$$
	\Delta u \ge u^\lambda
	\quad
	\mbox{in } B_1 \setminus \{ 0 \},
$$
where $B_1$ is a unit ball in $\mathbb R^n$, $n \ge 3$,
conditions~\eqref{1.5}--\eqref{1.7} with $q (r) = r^{\lambda (2 - n)}$ and $g (\zeta) = \zeta^\lambda$ yield
$$
	\lambda \ge \frac{n}{n - 2}.
$$
This coincides with the well-known B\'rezis--V\'eron condition~\cite{BV}.

For the more general inequality
\begin{equation}
	\Delta u \ge |x|^s u^\lambda
	\quad
	\mbox{in } B_1 \setminus \{ 0 \},
	\label{1.8}
\end{equation}
conditions~\eqref{1.5}--\eqref{1.7} with $q (r) = r^{\lambda (2 - n) + s}$ and $g (\zeta) = \zeta^\lambda$ take the form $\lambda > 1$ and $s \le \lambda (n - 2) - n$.

The results of~\cite{meCMFI} can obviously be applied to inequalities with non-power nonlinearities for an unknown function $u$. 
For example, in the case of the inequality
$$
	\Delta u \ge u^{n / (n - 2)} \log^\mu (1 + u)
	\quad
	\mbox{in } B_1 \setminus \{ 0 \},
$$
where $B_1$ is a unit ball in $\mathbb R^n$, $n \ge 3$,
it follows from~\eqref{1.5}--\eqref{1.7} with $q (r) = r^{- n} / \log (e^3 + r^{2 - n})$ and $g (\zeta) = \zeta^{n / (n - 2)} / \log (e^3 + \zeta)$ that any positive solution 
has a removable singularity at zero if $\mu \ge -1$. This is in good agreement with the result of paper~\cite{VV}.

Now, let us examine the critical exponent $\lambda = 1$ in the right-hand side of~\eqref{1.8}. Namely, consider the inequality
\begin{equation}
	\Delta u \ge |x|^s u \log^\mu (1 + u)
	\quad
	\mbox{in } B_1 \setminus \{ 0 \}.
	\label{1.10}
\end{equation}
Putting $q (r) = r^{2 - n + s}$ and $g (\zeta) = \zeta \log^\mu (1 + \zeta)$, 
we obtain from \eqref{1.5}--\eqref{1.7} that any positive solution of~\eqref{1.10} has a removable singularity at zero if $s \le - 2$ and $\mu > 2$.

All the above conditions are exact.
It can easily be seen that~\eqref{1.5} is equivalent to~\eqref{1.3} with $m = 2$.
Condition~\eqref{1.3} is obviously not sufficient to remove a singularity since 
we must also take into account the behavior of the function $f$ in a neighborhood of zero.
Moreover, in the case of linear equations, we need to impose some restrictions on the growth of solutions~\cite{Marcus}.
However, it can be shown that condition~\eqref{1.3} is necessary. 
Namely, if~\eqref{1.3} is not valid, then inequality~\eqref{1.1} has a positive solution with a non-removable singularity at zero for any function $f \in L_{loc, \infty} (\overline{B_1} \setminus \{ 0 \})$.
The proof of this fact is the subject of Theorem~\ref{T2.1} presented to your attention.
 
\section{Main results}

\begin{theorem}\label{T2.1}
Suppose that
\begin{equation}
	\int_1^\infty
	g^{- 1 / m} (\zeta)
	\zeta^{1 / m - 1}
	\,
	d\zeta
	=
	\infty.
	\label{T2.1.1}
\end{equation}
Then~\eqref{1.1}
has a positive solution with a non-removable singularity at zero for any $f \in L_{loc, \infty} (\overline{B_1} \setminus \{ 0 \})$.
\end{theorem}

The proof of Theorem~\ref{T2.1} relies on the following statement.

\begin{theorem}\label{T2.2}
Let~\eqref{T2.1.1} be valid. Then the problem
\begin{equation}
	\left\{
		\begin{aligned}
			&
			\left(
				\frac{
					1
				}{
					r^{n-1}
				}
				\frac{d}{dr}
				\left(
					r^{n-1}
					\frac{d}{dr}
				\right)
			\right)^{m/2}
			w
			\ge
			F (r)
			g (1 + w)
			\quad
			\mbox{on } (0, 1],
			\\
			&
			w^{(i)} (1) = 0,
			\quad
			i = 0, 1, \ldots, m - 1.
		\end{aligned}
	\right.
	\label{T2.2.1}		
\end{equation}
has a solution on the entire interval $(0, 1]$ for any $F \in L_{loc, \infty} ((0, 1])$.
In so doing, this solution is a positive decreasing function on $(0, 1)$.
\end{theorem}

\begin{proof}
Without loss of generality, it can be assumed that $F$ is a positive non-increasing function on $(0, 1)$; otherwise we replace $F$ by 
$$
	F_{max} (r)
	=
	\max
	\left\{ 
		1,
		\operatorname*{ess\,sup}_{(r, 1)}
		F
	\right\}.
$$
We prove the theorem in several steps.
\paragraph{Step 1} 
Let us define mappings $A_i : C ((0, 1]) \to L_{loc, \infty} ((0, 1])$ as follows:
$$
	A_0 v (r) = F (r) g (1 +|v (r)|)
$$
and
$$
	A_i v (r)
	=
	\int_r^1
	\int_\rho^1
	\left(
		\frac{\xi}{\rho}
	\right)^{n - 1}
	A_{i - 1} v (\xi)
	d\xi
	d\rho,
	\quad
	i = 1, \ldots, m / 2,
$$

Assume that $v \in C ((0, 1])$ is a non-negative non-increasing function.
In this case, $A_i v$ are obviously positive non-increasing functions on $(0, 1)$.
We extend $A_i v$ on the whole set $(0, \infty)$ by putting $A_i v (r) = 0$ for all $r \in [1, \infty)$, $i = 0, \ldots, m / 2$.

Let $r > 0$ and $h > 0$ be some real numbers.
We have
\begin{align*}
	A_i v (r + h)
	&
	=
	\int_{r + h}^\infty
	\int_\rho^\infty
	\left(
		\frac{\xi}{\rho}
	\right)^{n - 1}
	A_{i-1} v (\xi)
	d\xi
	d\rho
	\\
	&
	\ge
	\int_{r + h}^{r + 2 h}
	\int_{r + 2 h}^\infty
	\left(
		\frac{\xi}{\rho}
	\right)^{n - 1}
	A_{i-1} v (\xi)
	d\xi
	d\rho,
\end{align*}
whence, in accordance with the estimate
$$
	\int_{r + h}^{r + 2 h}
	\int_{r + 2 h}^\infty
	\left(
		\frac{\xi}{\rho}
	\right)^{n - 1}
	A_{i-1} v (\xi)
	d\xi
	d\rho
	\ge
	\frac{
		h
	}{
		(r + 2 h)^{n - 1}
	}
	\int_{r + 2 h}^\infty
	\xi^{n - 1}
	A_{i-1} v (\xi)
	d\xi,
$$
it follows that
\begin{equation}
	A_i v (r + h)
	\ge
	\frac{
		h
	}{
		(r + 2 h)^{n - 1}
	}
	\int_{r + 2 h}^\infty
	\xi^{n - 1}
	A_{i-1} v (\xi)
	d\xi,
	\quad
	i = 1, \ldots, m / 2.
	\label{PT2.2.1}
\end{equation}
Since
\begin{align*}
	\int_{r + 2 h}^\infty
	\xi^{n - 1}
	A_{i-1} v (\xi)
	d\xi
	&
	{}
	\ge
	\int_{r + 2 h}^{r + 3 h}
	\xi^{n - 1}
	A_{i-1} v (\xi)
	d\xi
	\\
	&
	{}
	\ge
	h
	(r + 2 h)^{n - 1}
	A_{i-1} v (r + 3 h),
\end{align*}
formula~\eqref{PT2.2.1} yields
\begin{equation}
	A_i v (r + h)
	\ge
	h^2
	A_{i - 1} v (r + 3 h),
	\quad
	i = 1, \ldots, m / 2.
	\label{PT2.2.2}
\end{equation}
Iterating~\eqref{PT2.2.1}, we also obtain
\begin{equation}
	A_i v (r + h)
	\ge
	\frac{
		h^{2 j - 1}
	}{
		(r + 2 j h)^{n - 1}
	}
	\int_{r + 2 j h}^\infty
	\xi^{n - 1}
	A_{i - j} v (\xi)
	d\xi
	\label{PT2.2.3}
\end{equation}
for all $1 \le j \le i$, $	i = 1, \ldots, m / 2$.
Analogously,~\eqref{PT2.2.2} implies that
\begin{equation}
	A_i v (r + h)
	\ge
	h^{2 j}
	A_{i - j} v (r + (2j + 1) h)
	\label{PT2.2.4}
\end{equation}
for all $1 \le j \le i$, $	i = 1, \ldots, m / 2$.

At the same time,
\begin{align*}
	&
	A_i v (r) - A_i v (r + j h)
	=
	\int_r^{r + j h}
	\int_\rho^\infty
	\left(
		\frac{\xi}{\rho}
	\right)^{n - 1}
	A_{i-1} v (\xi)
	d\xi
	d\rho
	\\
	&
	\quad
	{}
	=
	\int_r^{r + j h}
	\int_\rho^{r + (j + 1) h}
	\left(
		\frac{\xi}{\rho}
	\right)^{n - 1}
	A_{i-1} v (\xi)
	d\xi
	d\rho
	\\
	&
	\quad
	\phantom{{}={}}
	{}
	+
	\int_r^{r + j h}
	\int_{r + (j + 1) h}^\infty
	\left(
		\frac{\xi}{\rho}
	\right)^{n - 1}
	A_{i-1} v (\xi)
	d\xi
	d\rho
\end{align*}
for all $1 \le j \le i - 1$, $i = 1, \ldots, m / 2$,
whence, in accordance with the inequalities
\begin{align*}
	&
	\int_r^{r + j h}
	\int_\rho^{r + (j + 1) h}
	\left(
		\frac{\xi}{\rho}
	\right)^{n - 1}
	A_{i-1} v (\xi)
	d\xi
	d\rho
	\\
	&
	\qquad
	{}
	\le
	j (j + 1)
	\left(
		\frac{r + (j + 1) h}{r}
	\right)^{n - 1}
	h^2
	A_{i - 1} v (r)
\end{align*}
and
$$
	\int_r^{r + j h}
	\int_{r + (j + 1) h}^\infty
	\left(
		\frac{\xi}{\rho}
	\right)^{n - 1}
	A_{i-1} v (\xi)
	d\xi
	d\rho
	\le
	\frac{
		j h
	}{
		r^{n - 1}
	}
	\int_{r + (j + 1) h}^\infty
	\xi^{n - 1}
	A_{i-1} v (\xi)
	d\xi,
$$
we have
\begin{align}
	A_i v (r) - A_i v (r + j h)
	\le
	{}
	&
	j (j + 1)
	\left(
		\frac{r + (j + 1) h}{r}
	\right)^{n - 1}
	h^2
	A_{i - 1} v (r)
	\label{PT2.2.5}
	\\
	&
	{}
	+
	\frac{
		j h
	}{
		r^{n - 1}
	}
	\int_{r + (j + 1) h}^\infty
	\xi^{n - 1}
	A_{i-1} v (\xi)
	d\xi
	\nonumber
\end{align}
for all $1 \le j \le i - 1$, $	i = 1, \ldots, m / 2$.

Now, assume that $0 < h < 2 r / m$. In this case,~\eqref{PT2.2.5} yields
\begin{equation}
	A_i v (r) - A_i v (r + j h)
	\le
	\lambda
	h^2
	A_{i - 1} v (r)
	+
	\frac{
		\mu h
	}{
		r^{n - 1}
	}
	\int_{r + (j + 1) h}^\infty
	\xi^{n - 1}
	A_{i-1} v (\xi)
	d\xi
	\label{PT2.2.6}
\end{equation}
for all $ \le j \le i - 1$, $i = 1, \ldots, m / 2$,
where
$$
	\lambda
	=
	\left(
		\frac{m}{2}
	\right)^2
	2^{n - 1}
	\quad
	\mbox{and}
	\quad
	\mu
	=
	\frac{m}{2}.
$$
In particular, we have
$$
	A_{m / 2} v (r) - A_{m / 2} v (r + h)
	\le
	\lambda
	h^2
	A_{m / 2 - 1} v (r)
	+
	\frac{
		\mu h
	}{
		r^{n - 1}
	}
	\int_{r + 2 h}^\infty
	\xi^{n - 1}
	A_{m / 2 - 1} v (\xi)
	d\xi.
$$
Combining this with the estimate
\begin{align*}
	A_{m / 2 - 1} v (r) 
	\le
	{}
	& 
	A_{m / 2 - 1} v (r + 3 h)
	+
	\lambda
	h^2
	A_{m / 2 - 2} v (r)
	\\
	&
	{}
	+
	\frac{
		\mu h
	}{
		r^{n - 1}
	}
	\int_{r + 4 h}^\infty
	\xi^{n - 1}
	A_{m / 2 - 2} v (\xi)
	d\xi
\end{align*}
which also follows from~\eqref{PT2.2.6}, we obtain
\begin{align*}
	&
	A_{m / 2} v (r) - A_{m / 2} v (r + h)
	\le
	\lambda
	h^2
	A_{m / 2 - 1} v (r + 3 h)
	+
	\lambda^2
	h^4
	A_{m / 2 - 2} v (r)
	\\
	&
	\quad
	{}
	+
	\frac{
		\mu h
	}{
		r^{n - 1}
	}
	\int_{r + 2 h}^\infty
	\xi^{n - 1}
	A_{m / 2 - 1} v (\xi)
	d\xi
	+
	\frac{
		\lambda 
		\mu 
		h^3
	}{
		r^{n - 1}
	}
	\int_{r + 4 h}^\infty
	\xi^{n - 1}
	A_{m / 2 - 2} v (\xi)
	d\xi.
\end{align*}
Estimating $A_{m / 2 - i} v (r)$ successively by~\eqref{PT2.2.6} with $i = 2, \ldots, m / 2 - 1$ and $j = 2 i + 1$, we arrive at the inequality
\begin{align}
	&
	A_{m / 2} v (r) - A_{m / 2} v (r + h)
	\le
	\sum_{i = 1}^{m / 2 - 1}
	\lambda^i h^{2 i}
	A_{m / 2 - i} v (r + (2 i + 1) h)
	\label{PT2.2.7}
	\\
	&
	\qquad
	{}
	+
	\lambda^{m / 2}
	h^m
	A_0 v (r)
	+
	\sum_{i = 1}^{m / 2}
	\frac{
		\lambda^{i - 1}
		\mu
		h^{2 i - 1}
	}{
		r^{n - 1}
	}
	\int_{r + 2 i h}^\infty
	\xi^{n - 1}
	A_{m / 2 - i} v (\xi)
	d\xi.
	\nonumber
\end{align}
From~\eqref{PT2.2.3}, it follows that
\begin{align*}
	A_{m / 2} v (r + h)
	&
	{}
	\ge
	\frac{
		h^{2 i - 1}
	}{
		(r + 2 i h)^{n - 1}
	}
	\int_{r + 2 i h}^\infty
	\xi^{n - 1}
	A_{m / 2 - i} v (\xi)
	d\xi
	\\
	&
	{}
	\ge
	\frac{
		h^{2 i - 1}
	}{
		2^{n - 1}
		r^{n - 1}
	}
	\int_{r + 2 i h}^\infty
	\xi^{n - 1}
	A_{m / 2 - i} v (\xi)
	d\xi,
	\quad
	1 \le i \le m / 2,
\end{align*}
Analogously,~\eqref{PT2.2.4} yields
$$
	A_{m / 2} v (r + h)
	\ge
	h^{2 i}
	A_{m / 2 - i} v (r + (2 i + 1) h),
	\quad
	i = 1, \ldots, m / 2.
$$
Thus,~\eqref{PT2.2.7} allows us to assert that
\begin{align}
	A_{m / 2} v (r)
	\le
	{}
	&
	\left(
		1
		+
		\sum_{i = 1}^{m / 2 - 1}
		\lambda^i
		+
		2^{n - 1}
		\mu
		\sum_{i = 1}^{m / 2}
		\lambda^{i - 1}
	\right)
	A_{m / 2} v (r + h)
	\label{PT2.2.8}
	\\
	&
	{}
	+
	\lambda^{m / 2}
	h^m
	A_0 v (r)
	\nonumber
	\\
	{}
	\le
	{}
	&
	\frac{
		\lambda^{m / 2} - 1
	}{
		\lambda - 1
	}
	(1 + 2^{n - 1} \mu)
	A_{m / 2} v (r + h)
	+
	\lambda^{m / 2}
	h^m
	A_0 v (r).
	\nonumber
\end{align}

\paragraph{Step 2}
We take
\begin{equation}
	\alpha 
	= 
	\max
	\left\{
		2^{(n - 1) / 2}
		m,
		2^{1 / m}
		\lambda^{1 / 2}
	\right\}
	\quad
	\mbox{and}
	\quad
	\beta 
	=
	\frac{
		\lambda^{m / 2} - 1
	}{
		\lambda - 1
	}
	(2 + 2^n \mu).
	\label{PT2.2.9}
\end{equation}
It can easily be seen that $\beta A_{m/2}$ is a completely continuous nonlinear operator in the space 
$
	W 
	= 
	\{
		w \in C ([r_0, 1])
		:
		\| w \|_{
			C ([r_0, 1])
		}
		\le
		1
	\}
$
for some real number $1 / 2 < r_0 < 1$.
Thus, by the Bohl--Brouwer fixed point theorem, there exists a function $w_0 \in W$ such that 
\begin{equation}
	w_0 (r)
	=
	\beta 
	A_{m / 2} w_0 (r)
	\label{PT2.2.10}
\end{equation}
for all $r \in [r_0, 1]$.
It obvious that  $w_0 (1) = 0$ and, moreover, $w_0$ is a positive decreasing function on the interval $[r_0, 1)$. In particular, we have $w_0 (r_0) > 0$.

Further, let $v_0$ be a solution of the Cauchy problem
\begin{equation}
	\left\{
		\begin{aligned}
			&
			- \frac{
				d
			}{
				d r
			}
			v_0^{1 / m} (r)
			=
			\frac{
				\alpha
			}{
				r^{(n + 1) / 2}
			}
			F^{1 / m} 
			\left(
				\frac{r}{2}
			\right)
			g^{1 / m} (1 + \beta v_0)
			\quad
			\mbox{on } (0, r_0],
			\\
			&
			v_0 (r_0) = w_0 (r_0).
		\end{aligned}
	\right.
	\label{PT2.2.11}
\end{equation}
To verify that~\eqref{PT2.2.11} has a solution on the entire interval $(0, r_0]$, we denote $u_0 = v_0^{1 / m}$. Then~\eqref{PT2.2.11} takes the form
$$
	\left\{
		\begin{aligned}
			&
			- \frac{
			d
			}{
				d r
			}
			u_0 (r)
			=
			\frac{
				\alpha
			}{
				r^{(n + 1) / 2}
			}
			F^{1 / m}
			\left(
				\frac{r}{2}
			\right)
			g^{1 / m} (1 + \beta u_0^m)
			\quad
			\mbox{on } (0, r_0],
			\\
			&
			u_0 (r_0) = w_0^{1 / m} (r_0).
		\end{aligned}
	\right.
$$
Integrating this, we obtain
$$
	\int_{
		u_0 (r_0)
	}^{
		u_0 (r)
	}
	g^{- 1 / m} (1 + \beta t^m)
	d t
	=
	\alpha
	\int_r^{r_0}
	F^{1 / m}
	\left(
		\frac{\xi}{2}
	\right)
	\frac{
		d\xi
	}{
		\xi^{(n + 1) / 2}
	},
$$
whence after the changing of variables $\zeta = 1 + \beta t^m$, it follows that
$$
	\frac{
		1
	}{
		m \beta^{1 / m}
	}
	\int_{
		1 + \beta u_0^m (r_0)
	}^{
		1 + \beta u_0^m (r)
	}
	g^{- 1 / m} (\zeta)
	(\zeta - 1)^{1 / m - 1}
	d \zeta
	=
	\alpha
	\int_r^{r_0}
	F^{1 / m}
	\left(
		\frac{\xi}{2}
	\right)
	\frac{
		d\xi
	}{
		\xi^{(n + 1) / 2}
	}.
$$
In view of~\eqref{T2.1.1}, the left-hand side of the last equality tends to infinity as $u_0 (r) \to \infty$. This, in turn, means that the solution $v_0$ of Cauchy problem~\eqref{PT2.2.11} can be defined for all $r \in (0, r_0]$.
In so doing, $v_0$ is a positive decreasing function on $(0, r_0]$.

\paragraph{Step 3}
We put
$$
	v (r)
	=
	\left\{
		\begin{aligned}
			&
			v_0 (r),
			&&
			r \in (0, r_0),
			\\
			&
			w_0 (r),
			&&
			r \in [r_0, 1].
		\end{aligned}
	\right.
$$
It is easy to see that $v$ is a positive decreasing function on the interval $(0, 1)$. By induction on $m$, we also have
\begin{equation}
	A_{m/2} v (r)
	\le
	\frac{
		1
	}{
		r^{(n - 1) m / 2}
	}
	F (r) 
	g (1 + v (r))
	\label{PT2.2.13}
\end{equation}
for all $r \in (0, 1]$.
Let us prove that $v$ satisfies the inequality
\begin{equation}
	v (r)
	\ge
	A_{m/2} v (r)
	\label{PT2.2.12}
\end{equation}
for all $r \in (0, 1]$.
If $r \in [r_0, 1]$, then~\eqref{PT2.2.12} immediately follows from~\eqref{PT2.2.10}.

Assume that $r \in (0, r_0)$ and $k$ is the minimal positive integer satisfying the relation $v (r) \le \beta^k v (r_0)$. 
We prove~\eqref{PT2.2.12} by induction on $k$.
At first, let $k = 1$ and $h = r_0 - r \ge 2 r / m$.
Integrating~\eqref{PT2.2.11}, we have
$$
	v_0^{1 / m} (r) - v_0^{1 / m} (r_*)
	=
	\int_r^{r_*}
	\frac{
		\alpha
	}{
		\xi^{(n + 1) / 2}
	}
	F^{1 / m} 
	\left(
		\frac{\xi}{2}
	\right)
	g^{1 / m} (1 + \beta v_0)
	d\xi,
$$
where $r_* = \min \{r_0, 2 r \}$,
whence it follows that
\begin{align*}
	v (r)
	&
	{}
	\ge
	\left(
		\int_r^{r_*}
		\frac{
			\alpha
		}{
			\xi^{(n + 1) / 2}
		}
		F^{1 / m} 
		\left(
			\frac{\xi}{2}
		\right)
		g^{1 / m} (1 + \beta v)
		d\xi
	\right)^m
	\\
	&
	{}
	\ge
	\left(
		\frac{
			\alpha 
			(r_* - r)
		}{
			r_*^{(n + 1) / 2}
		}
		F^{1 / m} 
		\left(
			\frac{r_*}{2}
		\right)
		g^{1 / m} (1 + \beta v (r_*))
	\right)^m.	
\end{align*}
Since $r_* - r \ge 2 r / m$, $r_* \le  2 r$, and $\beta v (r_*) \ge \beta v (r_0) \ge v (r)$, this implies the inequality
\begin{equation}
	v (r)
	\ge
	\left(
		\frac{
			\alpha
		}{
			2^{(n - 1) / 2}
			m
		}
	\right)^m
	\frac{
		1
	}{
		r^{(n - 1) m / 2}
	}
	F (r)
	g (1 + v (r)).
	\label{PT2.2.14}
\end{equation}
In its turn, combining~\eqref{PT2.2.14} with~\eqref{PT2.2.9} and~\eqref{PT2.2.13}, we obtain~\eqref{PT2.2.12}.

Now, let $k = 1$ and $h = r_0 - r < 2 r / m$.
From~\eqref{PT2.2.8}, it follows that
\begin{equation}
	A_{m / 2} v (r)
	\le
	\frac{
		\lambda^{m / 2} - 1
	}{
		\lambda - 1
	}
	(1 + 2^{n - 1} \mu)
	A_{m / 2} v (r_0)
	+
	\lambda^{m / 2}
	h^m
	A_0 v (r).
	\label{PT2.2.15}
\end{equation}
In so doing, by~\eqref{PT2.2.10}, we have
$$
	v (r) \ge v (r_0) = \beta A_{m / 2} v (r_0).
$$
In view of~\eqref{PT2.2.9}, this yields
\begin{equation}
	\frac{1}{2}
	v (r)
	\ge
	\frac{
		\lambda^{m / 2} - 1
	}{
		\lambda - 1
	}
	(1 + 2^{n - 1} \mu)
	A_{m / 2} v (r_0).
	\label{PT2.2.16}
\end{equation}
Integrating~\eqref{PT2.2.11}, we also obtain
$$
	v_0^{1 / m} (r) - v_0^{1 / m} (r_0)
	=
	\int_r^{r_0}
	\frac{
		\alpha
	}{
		\xi^{(n + 1) / 2}
	}
	F^{1 / m} 
	\left(
		\frac{\xi}{2}
	\right)
	g^{1 / m} (1 + \beta v_0)
	d\xi,
$$
whence it follows that
\begin{align*}
	v (r)
	&
	{}
	\ge
	\left(
		\int_r^{r_0}
		\frac{
			\alpha
		}{
			\xi^{(n + 1) / 2}
		}
		F^{1 / m} 
		\left(
			\frac{\xi}{2}
		\right)
		g^{1 / m} (1 + \beta v)
		d\xi
	\right)^m
	\\
	&
	{}
	\ge
	\left(
		\frac{
			\alpha 
			h
		}{
			r_0^{(n + 1) / 2}
		}
		F^{1 / m} 
		\left(
			\frac{r_0}{2}
		\right)
		g^{1 / m} (1 + \beta v (r_0))
	\right)^m.	
\end{align*}
Since $0 < r_0 < 1$, $r_0 / 2 = (r + h) / 2 < r$, and $\beta v (r_0) \ge v (r)$, this implies the estimate
$$
	v (r)
	\ge
	\alpha^m
	h^m
	F (r)
	g (1 + v (r))
	=
	\alpha^m
	h^m
	A_0 v (r);
$$
therefore, taking into account~\eqref{PT2.2.9}, one can assert that
\begin{equation}
	\frac{1}{2}
	v (r)
	\ge
	\lambda^{m / 2}
	h^m
	A_0 v (r).
	\label{PT2.2.17}
\end{equation}
Finally, combining the last inequality with~\eqref{PT2.2.15} and~\eqref{PT2.2.16}, we arrive at~\eqref{PT2.2.12}.

Assume further that~\eqref{PT2.2.12} is valid for all $k \le l$, where $l$ is a positive integer. Let us show that~\eqref{PT2.2.12} is also valid for $k = l + 1$. We take a real number $r < r_1 < r_0$ satisfying the condition $v (r) = \beta v (r_1)$.
By the induction hypothesis, we obviously have
\begin{equation}
	v (r_1) \ge A_{m / 2} v (r_1).
	\label{PT2.2.18}
\end{equation}

If $h = r - r_1 < 2 r / m$, then repeating the above argument with $r_0$ replaced by $r_1$, we obtain~\eqref{PT2.2.14}. 
In view~\eqref{PT2.2.9} and~\eqref{PT2.2.13}, this immediately implies~\eqref{PT2.2.12}.
Let $h = r - r_1 \ge 2 r / m$.
In this case, repeating the argument given in the proof of~\eqref{PT2.2.15} with $r_0$ replaced by $r_1$, we have
\begin{equation}
	A_{m / 2} v (r)
	\le
	\frac{
		\lambda^{m / 2} - 1
	}{
		\lambda - 1
	}
	(1 + 2^{n - 1} \mu)
	A_{m / 2} v (r_1)
	+
	\lambda^{m / 2}
	h^m
	A_0 v (r).
	\label{PT2.2.19}
\end{equation}
At the same time,~\eqref{PT2.2.18} allows us to assert that
$$
	v (r) = \beta v (r_1) \ge \beta A_{m / 2} v (r_1),
$$
whence, in accordance with~\eqref{PT2.2.9}, we obtain
\begin{equation}
	\frac{1}{2}
	v (r)
	\ge
	\frac{
		\lambda^{m / 2} - 1
	}{
		\lambda - 1
	}
	(1 + 2^{n - 1} \mu)
	A_{m / 2} v (r_1).
	\label{PT2.2.20}
\end{equation}
It is also clear that the proof of~\eqref{PT2.2.17} remains valid if $r_0$ is replaced by $r_1$.
Thus, combining~\eqref{PT2.2.19} with~\eqref{PT2.2.17} and~\eqref{PT2.2.20}, we again arrive at~\eqref{PT2.2.12}.

\paragraph{Step 4}
Let us prove that the equation
\begin{equation}
	w (r) = A_{m / 2} w (r)
	\label{PT2.2.21}
\end{equation}
has a solution on the entire interval $(0, 1]$. 
Indeed, we put $w_1 = 0$ and
$$
	w_{i + 1} (r) = A_{m / 2} w_i (r),
	\quad
	i = 1, 2, \ldots.
$$
In view of~\eqref{PT2.2.12}, by induction on $i$, it can be shown that 
$$
	0 \le w_i (r) \le w_{i + 1} (r) \le v (r) 
	\quad
	\mbox{for all } r \in (0, 1], 
	\quad
	i = 1, 2, \ldots.
$$
Thus, there is a function $w : (0, 1] \to [0, \infty)$ such that $w_i (r) \to w (r)$ as $i \to \infty$ for all $r \in (0, 1]$.
By Lebesgue's dominated convergence theorem, this function satisfies equation~\eqref{PT2.2.21} on the entire interval $(0, 1]$.
To complete the proof, it remains to verify by direct differentiation that $w$ is a solution of problem~\eqref{T2.2.1}.
\end{proof} 

\begin{proof}[Proof of Theorem~\ref{T2.1}]
We put 
$$
	F (r)
	=
	\frac{
		1
	}{
		r^{m + n}
	}
	+
	\operatorname*{ess\,sup}_{
		B_1 \setminus B_{r / 2}
	}
	|f|.
$$
By Theorem~\ref{T2.2}, there exists a solution $w : (0, 1] \to [0, \infty)$ of problem~\eqref{T2.2.1} which is a positive decreasing function on the interval $(0, 1)$.
Representing the Laplace operator in multidimensional polar coordinates, it is easy to verify that
$$
	u (x)
	=
	w (|x|)
$$
is a solution of inequality~\eqref{1.1}.
Integrating~\eqref{T2.2.1}, we obviously arrive at~\eqref{PT2.2.21}, whence it follows that
$$
	w (r)
	\ge
	\int_r^{2 r}
	d\rho_1
	\int_{\rho_1}^{2 r}
	d\rho_2 
	\int_{\rho_2}^{2 r}
	d\rho_3
	\ldots
	\int_{\rho_{m - 1}}^{2 r}
	F (\rho_m) g (1 + w)
	d\rho_m
	\ge
	r^m
	F (2 r)
	g (1)
$$
for all $r \in (0, 1 / 2)$. This, in turn, yields
$$
	u (x)
	\ge
	\frac{
		g (1)
	}{
		2^{m + n}
		|x|^n
	}
$$
for all $x \in B_{1 / 2} \setminus \{ 0 \}$. Thus, $u \not\in L_1 (B_1)$, i.e., $u$ has a non-removable singularity at zero.
\end{proof}

\end{document}